\documentclass[times]{article}

\usepackage{textcomp} 
\usepackage{pifont}   
\usepackage{booktabs}
\usepackage{amssymb}
\usepackage{amsthm}
\usepackage{amsmath,amscd}

\usepackage{latexsym,enumerate,tikz,  ifthen, mathrsfs, mathtools, cite}

\usepackage{bbm}
\usepackage{hyperref}

\DeclareRobustCommand{\VAN}[3]{#2} 
\DeclareRobustCommand{\DELA}[3]{#2} 

\newtheorem{definition}{Definition}
\newtheorem{theorem}{Theorem}
\newtheorem{corollary}{Corollary}

\begin{document}

\title{$N\times N$ matrix time--band limiting examples}

\author{Bruno Eijsvoogel}
\date{}

\maketitle

\begin{abstract}%
Time-- and band--limiting in the context of orthogonal
polynomials has been studied since the 1980's. It involves
finding differential or difference operators with special
commutative properties.
More recently this topic has been generalized to the cases
were the orthogonal polynomials are matrix-valued. Leading
to differential or difference operators with matrix coefficients.
So far the only explicit examples of such operators were $2\times 2$
matrices. In this paper we give a number of $N\times N$ examples and
a counterexample to illustrate the role that strong Pearson equations
can play in finding such examples.
\end{abstract}

\section{Introduction}

In this paper we examine a number of $N\times N$ matrix-valued orthogonal
polynomials (MVOP) which are examples that fit into the noncommutative 
time--band limiting framework introduced
in \cite{GPZ}. So in this introduction we will first mention some of 
the work that led up to \cite{GPZ}.

\subsection{Time--Band limiting}
In the 1960's Slepian, Landau and Pollack published a series of seminal
papers \cite{Pswf1, Pswf2, Pswf3, Pswf4} studying time--limiting and
band--limiting in the context of the Fourier expansion.
The results they were able to derive depended
strongly on a seemingly miraculous commutation between a certain 
integral operator
\begin{equation}\label{eq:integral_op}
(K \cdot \phi)(x)
=
\int_{-1}^{1}\frac{\sin(c(x-y))}{\pi(x-y)} \phi(y)dy
\end{equation} 
and a differential operator,
$$
(D \cdot \phi)(x) = (1-x^2)\phi''(x) -2x\phi'(x) -c^2 x^2\phi(x)
$$
and in particular on their sharing of eigenfunctions. 

This miraculous commutation seemed to be more than just a lucky fluke
since various generalizations of this situation resulted in useful
commuting differential operators. We recommend \cite{slepianhelp} for 
a very nice introduction to the topic.

Next we will briefly discuss bispectrality, what the
connection is with time--band limiting and some noncommutative
generalizations that have been made. 

\subsection{Bispectrality}
Although bispectrality started out \cite{DG86} and still is a topic 
within integrable systems, a discussion of these roots is beyond the
scope of this introduction.

The notion of bispectrality usually involves
two (differential or difference) operators. One we call $D$
acting on a variable $x$ and another will be called $L$
acting on another variables $y$. Lastly we need a so-called
bispectral function
$\phi(x,y)$ that is an eigenfunction of both
$$
(D \cdot \phi)(x,y) = \lambda(y)\phi(x,y),
\qquad
(L \cdot \phi)(x,y) = \mu(x)\phi(x,y).
$$

One of the most basic examples of a bispectral function is
$$
\phi(x,y) = e^{-2\pi i x y}, 
\qquad  \qquad
L = \partial_y, 
\qquad
D = \partial_x.
$$
This is also the kernel of the Fourier transform and
the kernel of the integral operator \eqref{eq:integral_op}
comes from the expression
$$
\frac{\sin(\pi T(x-y))}{\pi(x-y)} = \int_{-T/2}^{T/2} e^{2\pi i x t}e^{-2\pi i y t}dt.
$$
This may not seem like a strong link between time--band limiting
and bispectrality at first, given the ubiquity of the exponential
function as well as that of Fourier analysis. However in \cite{GYprol}
certain types of integral kernels of the form
$$
\mathscr{K}(x,y)
=
\int_{\Gamma} \phi(x,z)\phi(y,z) dz,
$$
were found to have differential operators commuting with their integral
operators, where the $\phi$ are bispectral functions and $\Gamma$ is 
a contour in $\mathbb{C}$. After that, similar results were found in \cite{Casperprol} of the form
$$
\mathscr{K}(x,y)
=
\int_{\Gamma} \phi(x,z)\phi^\dagger(y,z) dz,
$$
were $\phi$ and $\phi^\dagger$ are bispectral functions associated to
KP hierarchy\footnote{The integrable hierarchy associated to
the Kadomtsev-Petviashvili equation.}. It should be noted though that the link between time--band
limiting and bispectrality was suggested already in \cite{DG86} at
bispectrality's inception.

In an effort to better understand the miraculous existence of
a differential operator with nice commutative properties
such as the one in \eqref{eq:integral_op}, a lot of work was
done to generalize this kind of analysis to different settings.
See for instance \cite{GLP}.

A sequence of classical orthogonal polynomials $(p_n(x))_n$ taken 
as a whole is also a bispectral function if we view the index as a 
discrete variable  $y=n\in\mathbb{N}_0$.
We can then define the operator $\delta$ as 
acting on a sequence $u_n$. We define it as a shift 
$(\delta^j\cdot u)_n = u_{n+j}$ as long as $n+j \geq 0$
and define any entries with negative index to equal $0$.
Then for $\phi(x,y)=p_n(x)$ the second eigenvalue equation would be
the three term recurrence relation
$$
(L\cdot p)_n(x) = xp_n(x), \qquad
L = a_n \delta + b_n + c_n \delta^{-1} ,
$$
and the first would be either a differential or difference operator
acting on the variable $x$ for which the polynomials are eigenfunctions.
The second operator is guaranteed to exist because we are dealing with
\textit{classical} orthogonal polynomials.

Casting classical orthogonal polynomials in the time--band 
limiting framework 
started with \cite{GLP, Grunnew} for continuous orthogonal polynomials,
and for the discrete case in \cite{perlchop, perl, perline}.
In these works analogous differential operators 
(or difference operators in the discrete case) were found that commuted 
with the analogous integral operator. Most of these
papers contained, in addition to the construction of the operators,
some analysis of the spectrum. For analysis of the eigenfunctions we
refer the reader to \cite{perleigen}.

\subsection{Noncommutative bispectrality and time--band limiting}

A noncommutative version of bispectrality was introduced
in \cite{GI03} with the goal of studying noncommutative
analogues of integrable lattice equations. 

Bispectrality in the noncommutative context involves
pairs of operators where one acts from the
left and the other from the right. It was also this thinking
that eventually led to the formalism of the Fourier algebras
which play a crucial role in \cite{Casper2020}.

This has been a motivation to pose a noncommutative
versions of time--band limiting problems in particular
those involving matrix-valued orthogonal polynomials \cite{GPZ15, GPZint}.
More work was done in this direction \cite{furthercastro, castrodarboux, GPZ} and it is \cite{GPZ} that we will build on further. 

The authors of \cite{GPZ} 
show how to construct 
matrix-valued differential operators that commute with
a corresponding integral operator with a matrix-valued
kernel. Various explicit examples that fit well into this
framework have been studied before, but these were always
matrices of size $2\times 2$. The goal of this paper is
to give $N\times N$ examples and to show how the strong Pearson
equations help to satisfy the conditions that are unique to the
matrix case. As a contrasting case we also treat an example in 
which the associated matrix weight does not satisfy strong 
Pearson equations and we subsequently see that this case only 
works for  $N=2$ and not for a number of larger matrix sizes.

\subsection{Overview}
The rest of this paper is structured as follows. 
Section \ref{sec:preTBL} reviews
the bare minimum of the noncommutative time--band limiting framework
that we will need for the results in Section \ref{sec:SP}.
In Section \ref{sec:WPO} we highlight some similarities as well as
differences between the different weights and operators
that will appear in the remainder of the paper.
We then apply this time--band limiting framework
to MVOP of Hermite--type, Laguerre--type,
Gegenbauer--type and Charlier--type in Section \ref{sec:SP}.
We also consider a slightly different Hermite--type scenario in 
Section \ref{sec:counter}, to which 
the framework is applicable for matrix size $N=2$ but not for some
larger matrix sizes. Finally the appendix has a double role. We use it 
to collect some explicit expressions of matrix weights and 
related quantities that appear in the proofs, but also to correct some
typos that have appeared in previous work.

\section{Noncommutative Time--Band Limiting for MVOP}
\label{sec:preTBL}
In this section we will very briefly recall the parts
of \cite{GPZ} that we will need to treat the examples that are
the main focus of this paper.

We consider a $N\times N$ matrix weight $W$ and its \emph{monic} MVOP $(P_n)_n$ with squared norms $\mathcal{H}_n$.
We then assume that the $P_n$ are simultaneous eigenfunctions of 
a $W$-symmetric second order differential operator $D$,
acting from the right and with eigenvalue matrix $\Lambda_n$ on 
the left
$$
P_n \cdot D = \Lambda_n P_n.
$$

\begin{definition}
Given $M \in \mathbb{N}_0$ and a matrix weight $W$, we define
the time-limiting operator $\chi_T^{(M)}$ to act as
$$
\left(F\cdot \chi_T^{(M)} \right)(x) = \sum_{n=0}^M \langle F, P_n\rangle \mathcal{H}_{n}^{-1} \, P_n(x)
$$
for matrix functions $F$ for which the matrix inner product 
$\langle F, P_n\rangle$ corresponding to $W$ is well-defined 
and has finite entries.

Given $\Omega\in\mathbb{R}$ the band--limiting operator
$\chi_B^{(\Omega)}$ acts
by multiplication of a
characteristic function
$(F\cdot\chi_B^{(\Omega)})(x)=
F(x)\mathbb{1}_{(-\infty,\Omega)}(x)$.
\end{definition}

\paragraph{Remark}
The subscripts for $\chi_T$ and $\chi_B$ refer to "time"
and "band". The variable $n$ plays the role of time,
$x$ plays the role of the spectral variable and each $\chi$
limits its respective variable.

\paragraph{Remark}
Note that the band--limiting operator $\chi_B^{(\Omega)}$
is the same no matter which MVOP are under consideration,
whereas this is not the case for the time--limiting operator
$\chi_T^{(M)}$.

One of the main results from \cite{GPZ} is the construction
of a differential operator $\mathcal{T}$ that commutes with both the
time-- and band--limiting operators separately.
A central requirement given in \cite[Equation (9)]{GPZ} is 
that we must be able to find a matrix $\mathcal{R}$ which
does not depend on $x$ or $\Omega$ and satisfies
\begin{equation}\label{eq:eq9}
\left(
\mathcal{R} 
- 
x(\Lambda_M + \Lambda_{M+1})
\right) 
W(x)
=
W(x)
\left(
\mathcal{R} 
- 
x(\Lambda_M +\Lambda_{M+1} )
\right)^\ast .
\end{equation}
Note that $\mathcal{R}$ is allowed to depend on $M$
and any other parameters that the weight $W$ might
have. 

Finally equation (11a) in \cite{GPZ} gives the 
construction of $\mathcal{T}$ as
$$
\mathcal{T} 
= 
x D + D(x-2\Omega) 
- x(\Lambda_M +\Lambda_{M+1}) + \mathcal{R},
$$
where $\mathcal{T}$ acts from the right. Note that the
condition \eqref{eq:eq9} guarantees that $\mathcal{T}$
will be $W$-symmetric. This is because $x$ and $D$ are
$W$-symmetric and so $xD + D x$ and $-2\Omega D$ are
as well.

\paragraph{Remark}
Strictly speaking the results presented in \cite{GPZ} only apply to
$W$-symmetric \textit{differential} operators. However upon inspection, 
the proofs apply equally well to $W$-symmetric 
\textit{difference} operators. We will therefore apply it to the 
Charlier--type example in Section \ref{sec:charlierTBL}.

Each of the examples we treat in this paper have been introduced
in previous works \cite{IKR2,KR,KdlRR,EMR}, 
including the matrix weight and the $W$-symmetric
second order differential or difference operator. What is left for us to do
is find the matrix $\mathcal{R}$ that satisfies \eqref{eq:eq9}
in order to construct $\mathcal{T}$.

\section{Weights, Parameters and Operators}
\label{sec:WPO}
The matrix weights discussed in this paper have
a number of similarities. They all have a free
parameter $\nu$.
They also all have some other parameters $(\alpha_j)_{j=1}^N$
and $(t_j^{(\nu)})_{j=1}^N$, which can be required to satisfy
a certain set of nonlinear equations. This requirement then
implies that the matrix weight satisfies a strong Pearson equation.
In the appendix we list some parameter values such that this requirement
is met. Only the matrix weight in Section \ref{sec:counter} does
\emph{not} need to satisfy this requirement as we do not need a strong
Pearson equation, and so we only require the
parameters $\alpha_j$ and $t_j^{(\nu)}$ to be positive real numbers.

Each matrix weight is given in its LDU decomposition
$$
W^{(\nu)}(x) = L^{(\nu)}(x)T^{(\nu)}(x)L^{(\nu)}(x)^\ast,
$$
where the diagonal matrix entries are 
$(T^{(\nu)}(x))_{jj}=t_j^{(\nu)} w_j^{(\nu)}(x)$ with 
each $w_j^{(\nu)}$ a corresponding scalar weight.
For example for the Hermite--type case 
$w_j^{(\nu)}(x)=e^{-x^2}$ does not depend on $j$ or $\nu$ but for the Gegenbauer--type case $w_j^{(\nu)}(x)=(1-x^2)^{\nu+j-1/2}$.

The lower triangular matrix has nonzero entries of the form $L^{(\nu)}(x)_{jk} = \frac{\alpha_j}{\alpha_k} p_{j-k}^{(\nu+k)}(x)$ where $p_n^{(\nu)}$ are the corresponding scalar orthogonal
polynomials and where the parameter $\nu+k$ shifts with the column index
only when it is appropriate\footnote{The scalar Hermite polynomials do not
have any such parameter so unsurprisingly in that case $L^{(\nu)}(x)=L(x)$. However, this is also true for the Charlier polynomials who do have
a parameter that could have been shifted. This is not done because their scalar ladder relations do not involve shifting this parameter.}.

Even without the restriction on the parameters $\alpha_j$ and $t_j^{(\nu)}$,
each matrix weight already has at least one $W^{(\nu)}$-symmetric
second order differential or difference operator with the MVOP as
simultaneous eigenfunctions. When we impose the requirements to obtain
strong Pearson equations we get \emph{an additional} operator with
these properties. These additional operators are the ones that appear
in Sections \ref{sec:SP} and the other kind of
operator is studied in Section \ref{sec:counter} which, as we will see,
does not fit into the framework of \cite{GPZ} as nicely.

\paragraph{Remark} \label{rmk:N}
The matrix weights we consider in this paper are all of the form
$$
W(x) = w(x)Q(x),
$$
where $w$ is a scalar classical weight and $Q$
is a matrix polynomial of degree $2N-2$. This means that \eqref{eq:eq9}
in these cases will always be a matrix polynomial equation. Or since
we are looking to solve for the entries of $\mathcal{R}$, \eqref{eq:eq9}
is a inhomogeneous linear system. We point this out to note that
roughly speaking, as $N$ grows, the number of equations for the 
entries of $\mathcal{R}$ grows as $N^3$ whereas the number of parameters
obviously is just $N^2$. So we conclude that cases where we \emph{can}
find $\mathcal{R}$ for all $N\in\mathbb{N}$ are far from generic.

\section{Examples with strong Pearson equations}
\label{sec:SP}
\subsection{Hermite, Laguerre and Gegenbauer}
\label{sec:HLG}
In this section we discuss three examples which correspond to
Hermite--, Laguerre-- and Gegenbauer--type MVOP introduced in
\cite{IKR2}, \cite{KR} and \cite{KdlRR} respectively. We discuss
them at the same time due to their close similarity. 
All three of these
cases have a parameter\footnote{The parameter $\nu$ is usually taken to be
positive real though it is possible to extend it in some cases.} family of weights $W^{(\nu)}$ that satisfy
certain requirements we will call
\textit{strong Pearson equations}
$$
\biggl\{
\begin{array}{c}
W^{(\nu+1)}(x) = W^{(\nu)}(x)\Phi^{(\nu)}(x), \\
W^{(\nu+1)\prime}(x) = W^{(\nu)}(x)\Psi^{(\nu)}(x),
\end{array}
$$
where $\Phi^{(\nu)}$ and $\Psi^{(\nu)}$ are matrix polynomials
of degree $\leq 2$ and exactly equal to $1$ respectively
$$
\Phi^{(\nu)}(x) = x^2 \phi_2^{(\nu)}  + x \phi_1^{(\nu)} + \phi_0^{(\nu)},
\qquad
\Psi^{(\nu)}(x) = x \psi_1^{(\nu)} + \psi_0^{(\nu)}.
$$ 
Explicit expressions for these polynomials are listed in
the appendix.

One of the main consequences of the strong Pearson equations is
that the MVOP $P_n^{(\nu)}$ satisfy the eigenvalue equation
$P_n^{(\nu)} \cdot D^{(\nu)} = \Lambda_n^{(\nu)} P_n^{(\nu)}$
with
\begin{equation}\label{eq:DO}
D^{(\nu)} = \partial_x^2 \Phi^{(\nu)}(x)^\ast + \partial_x \Psi^{(\nu)}(x)^\ast ,
\qquad
\Lambda_n^{(\nu)} = n(n-1)\phi_2^{(\nu)\ast}+n \psi_1^{(\nu)\ast}.
\end{equation}
Another consequence is
\begin{equation}\label{eq:switching}
W^{(\nu)}(x)\Phi^{(\nu)}(x) = \Phi^{(\nu)}(x)^\ast W^{(\nu)}(x),
\qquad
W^{(\nu)}(x)\Psi^{(\nu)}(x) = \Psi^{(\nu)}(x)^\ast W^{(\nu)}(x),
\end{equation}
which follows from the fact that $W^{(\nu+1)}$ as well as 
$W^{(\nu+1)\prime}$ are symmetric matrices. The idea of the proof
of the following theorem is that the previous two equations
can be combined into an equation of the form of \eqref{eq:eq9}
and hence provide us with a matrix $\mathcal{R}^{(\nu)}$.

\begin{theorem} \label{thm:3}
Let $\mathcal{T}^{(\nu)}$ be the following matrix differential operator
$$
\mathcal{T}^{(\nu)}
=
x D^{(\nu)} + D^{(\nu)} (x
- 2\Omega)
- x(\Lambda_M^{(\nu)} +\Lambda_{M+1}^{(\nu)}) + \mathcal{R}^{(\nu)},
$$
as given in \cite[equation (11a)]{GPZ}.

\begin{itemize}
\item
When $D^{(\nu)}$ is as in the Hermite--type example \cite[Corollary 3.11]{IKR2}, then $\mathcal{R}_H^{(\nu)}=-(2M+1)\psi_0^{(\nu)\ast}$ satisfies
\eqref{eq:eq9}.
\item
When $D^{(\nu)}$ is as in the Laguerre--type example \cite[Corollary 6.3]{KR}, then $\mathcal{R}_L^{(\nu)}=-M^2\phi_1^{(\nu)\ast} -(2M+1)\psi_0^{(\nu)\ast}$ satisfies \eqref{eq:eq9}.
\item
When $D^{(\nu)}$ is as in the Gegenbauer--type example \cite[Corollary 2.5]{KdlRR} (denoted $\mathscr{D}^{(\nu)}$), then $\mathcal{R}_G^{(\nu)}=-\left(\frac{M^2}{2\nu+N} + 2M+1 \right)\psi_0^{(\nu)\ast}$ satisfies \eqref{eq:eq9}.
\end{itemize}
\end{theorem}
\begin{proof}
For the Hermite--type case $\Phi^{(\nu)}$ is a degree 1 polynomial, see
\eqref{eq:phihermite}. So by \eqref{eq:DO} we have 
$\Lambda_n^{(\nu)}=n\psi_1^{(\nu)\ast}$. 
This means that when
we choose $\mathcal{R}^{(\nu)}=-(2M+1)\psi_0^{(\nu)\ast}$ the degree
1 polynomial that appears in \eqref{eq:eq9} is
$$
\mathcal{R}^{(\nu)} - x(\Lambda_M^{(\nu)}+\Lambda_{M+1}^{(\nu)} )
=
-(2M+1)\Psi^{(\nu)}(x)^\ast.
$$
The second equation in \eqref{eq:switching} is then equivalent to the desired
condition in \eqref{eq:eq9}.

For the Laguerre case $\Phi^{(\nu)}$ is of degree 2 but $\phi_0^{(\nu)}=0$,
see \eqref{eq:philaguerre}.
So now $x^{-1}\Phi^{(\nu)}(x)=x\phi_2^{(\nu)}+\phi_1^{(\nu)}$ is a degree 1 polynomial.
We can leverage this and \eqref{eq:switching} to obtain
\begin{equation}
\label{eq:switchL}
\begin{aligned}
(x\phi_2^{(\nu)\ast}+\phi_1^{(\nu)\ast})W^{(\nu)}(x) = W^{(\nu)}(x)(x\phi_2^{(\nu)\ast}+\phi_1^{(\nu)\ast})^\ast,
\\
(x\psi_1^{(\nu)\ast}+\psi_0^{(\nu)\ast} )W^{(\nu)}(x) = W^{(\nu)}(x)(x\psi_1^{(\nu)\ast}+\psi_0^{(\nu)\ast})^\ast.
\end{aligned}
\end{equation}
Since in this case by \eqref{eq:DO} $\Lambda_n^{(\nu)}=n(n-1)\phi_2^{(\nu)\ast} + n \psi_1^{(\nu)\ast}$, we can use
$\mathcal{R}^{(\nu)} = -M^2\phi_1^{(\nu)\ast} -(2M+1)\psi_0^{(\nu)\ast} $
to get
$$
\mathcal{R}^{(\nu)} - x(\Lambda_M^{(\nu)}+\Lambda_{M+1}^{(\nu)} )
=
-\frac{M^2}{x}\Phi^{(\nu)}(x)^\ast -(2M+1)\Psi^{(\nu)}(x)^\ast.
$$
Using both equations in \eqref{eq:switchL} it can be seen that the condition 
in \eqref{eq:eq9} is satisfied.

For the Gegenbauer case we have the good fortune that the leading
coefficients of $\Phi^{(\nu)}$ and $\Psi^{(\nu)}$ are equal up to
scalar factor $\phi_2^{(\nu)} = \frac{1}{2\nu+N} \psi_1^{(\nu)}$,
see \eqref{eq:phigegenbauer}. 
So then the eigenvalue in \eqref{eq:DO} simplifies to
$\Lambda_n^{(\nu)}= (\frac{n(n-1)}{2\nu+N}+n)\psi_1^{(\nu)\ast}$,
and in a similar way to the Hermite case, we can set
$\mathcal{R}^{(\nu)} = -\left(\frac{M^2}{2\nu+N} + 2M+1 \right)\psi_0^{(\nu)\ast}$. The degree 1 polynomial that appears in \eqref{eq:eq9} is then
$$
\mathcal{R}^{(\nu)} - x(\Lambda_M^{(\nu)}+\Lambda_{M+1}^{(\nu)} )
=
-\left(\frac{M^2}{2\nu+N} + 2M+1 \right)\Psi^{(\nu)}(x)^\ast.
$$
which satisfies \eqref{eq:eq9} again due to the equation in \eqref{eq:switching} involving $\Psi^{(\nu)}$.

\end{proof}

In \cite{GPZ} the reason to construct an operator like $\mathcal{T}^{(\nu)}$
is its commutativity with the time-- and
band--limiting operators.
\begin{corollary}
Given the values $\Omega\in \mathbb{R}$ and $M\in\mathbb{N}_0$ 
the matrix differential operator $\mathcal{T}^{(\nu)}$ in Theorem \ref{thm:3}
constructed with the quantities in \cite[Corollary 3.11]{IKR2}
and with $\mathcal{R}^{(\nu)}=\mathcal{R}_H^{(\nu)}$, commutes with
the band--limiting operator $\chi_B^{(\Omega)}$ and the time--limiting
operator $\chi_T^{(M)}$ that corresponds to the weight $W^{(\nu)}$ described in \cite[Section 3.3]{IKR2}.
\end{corollary}
\begin{proof}
The proof follows immediately from the main results in \cite{GPZ}.
\end{proof}
Analogous results hold for the Laguerre-- and Gegenbauer--type examples.
In each example the band--limiting operator is always the same but the time--limiting operator is different for each case, because it involves
the matrix inner product and the MVOP.

\subsection{Charlier}
\label{sec:charlierTBL}

In this section we discuss the Charlier--type MVOP that were  introduced
 in \cite{EMR}.
Let us first recall some notation for
the forward and backwards finite shift operators
$$
\left(F\cdot \Delta \right)(x) = F(x+1)-F(x),
\qquad
\left(F\cdot \nabla \right)(x) = F(x)-F(x-1).
$$

As before we have a family of weights $W^{(\nu)}$ parametrized by
$\nu \in \mathbb{N}_0$ that satisfies
certain requirements which are discrete analogues to the 
\textit{strong Pearson equations}
$$
\Biggl\{
\begin{aligned}
W^{(\nu+1)}(x) &= W^{(\nu)}(x)\Phi^{(\nu)}(x), \\
(W^{(\nu+1)}\cdot\nabla )(x) &= W^{(\nu)}(x)\Psi^{(\nu)}(x),
\end{aligned}
$$
where $\Phi^{(\nu)}$ and $\Psi^{(\nu)}$ are matrix polynomials
$$
\Phi^{(\nu)}(x) = x^2 \phi_2^{(\nu)}  + x \phi_1^{(\nu)} + \phi_0^{(\nu)},
\qquad
\Psi^{(\nu)}(x) = x \psi_1^{(\nu)} + \psi_0^{(\nu)}.
$$ 
Explicit expressions for these polynomials were given in \cite{EMR} 
but are also listed in
the appendix.

As in Section \ref{sec:HLG} one of the main consequences of 
these strong Pearson equations is
that the MVOP $P_n^{(\nu)}$ satisfy the eigenvalue equation
$P_n^{(\nu)} \cdot D^{(\nu)} = \Lambda_n^{(\nu)} P_n^{(\nu)}$
but now with a difference operator (which is denoted $\Delta S^{(\lambda)}$
in \cite{EMR})
\begin{equation}\label{eq:DiffO_new}
D^{(\nu)} = - \Delta \nabla \Phi^{(\nu)}(x)^\ast - \nabla\Psi^{(\nu)}(x)^\ast ,
\qquad
\Lambda_n^{(\nu)} = - n(n-1)\phi_2^{(\nu)\ast} - n \psi_1^{(\nu)\ast}.
\end{equation}
Another consequence is
\begin{equation}\label{eq:switching_new}
W^{(\nu)}(x)\Phi^{(\nu)}(x) = \Phi^{(\nu)}(x)^\ast W^{(\nu)}(x),
\qquad
W^{(\nu)}(x)\Psi^{(\nu)}(x) = \Psi^{(\nu)}(x)^\ast W^{(\nu)}(x),
\end{equation}
which follows from the fact that $W^{(\nu+1)}$ is a symmetric matrix. 
The idea of the proof
of the following theorem is that the previous two equations
can be combined into an equation of the form of \eqref{eq:eq9}
and hence provide us with a matrix $\mathcal{R}^{(\nu)}$.

\begin{theorem} \label{thm:charlier}
Let $\mathcal{T}^{(\nu)}$ be the following matrix difference operator
$$
\mathcal{T}^{(\nu)}
=
x D^{(\nu)} + D^{(\nu)} (x
- 2\Omega)
- x(\Lambda_M^{(\nu)} +\Lambda_{M+1}^{(\nu)}) + \mathcal{R}^{(\nu)},
$$
as given in \cite[equation (11a)]{GPZ}.
When $D^{(\nu)}$ is as in the Charlier--type example \eqref{eq:DiffO_new}, then $\mathcal{R}_C^{(\nu)}=M^2 (\phi_1^{(\nu)\ast}-\psi_1^{(\nu)\ast})+(2M+1)\psi_0^{(\nu)\ast}$ satisfies \eqref{eq:eq9}.
\end{theorem}
\begin{proof}

We start off using \eqref{eq:DiffO_new} to show that
$\Lambda_M^{(\nu)} + \Lambda_{M+1}^{(\nu)} = -M^2 \phi_2^{(\nu)\ast}-(2M+1)\psi_1^{(\nu)\ast}$. 
For the Charlier--type case $\Phi^{(\nu)}$ and $\Psi^{(\nu)}$ have
the same constant term: $\psi_0^{(\nu)}=\phi_0^{(\nu)}$. This means
that $x^{-1}(\Phi^{(\nu)}(x) - \Psi^{(\nu)}(x)) = x\phi_2^{(\nu)} +\phi_1^{(\nu)}-\psi_1^{(\nu)}$ is a degree 1 polynomial.
Due to \eqref{eq:switching_new} we see that this polynomial has a desired
commutation property with $W^{(\nu)}$. If we now choose $\mathcal{R}_C^{(\nu)}=M^2 (\phi_1^{(\nu)\ast}-\psi_1^{(\nu)\ast})+(2M+1)\psi_0^{(\nu)\ast}$ we get that
$$
\mathcal{R}_C^{(\nu)} - x(\Lambda_M^{(\nu)} + \Lambda_{M+1}^{(\nu)})
=
\frac{M^2}{x}(\Phi^{(\nu)}(x)^\ast - \Psi^{(\nu)}(x)^\ast)+(2M+1) \Psi^{(\nu)}(x)^\ast,
$$
which then satisfies \eqref{eq:eq9}.
\end{proof}

\section{Counterexample}
\label{sec:counter}

In Section \ref{sec:HLG} we studied the Hermite--type matrix weight 
given in \cite[Section 3.3]{IKR2} that had certain restrictions 
on its parameters $\alpha_j$ and $t_j^{(\nu)}$.
In \cite[Section 3.2]{IKR2} however the weight
is described without these restrictions and so in this case one 
cannot derive the strong Pearson equations and the differential operator
that follows from them. This latter type of weight has also been studied
in \cite{DER}. Without the Pearson equations the parameter $\nu$
loses meaning so we drop that part of the notation.

Nevertheless there is \textit{another} 
$W$-symmetric second order
differential operator that has the MVOP as eigenfunctions. We denote
this operator with a slightly different normalization than in \cite[Proposition 3.5]{IKR2} as
$$
D
=
-\tfrac12 \partial_x^2 + \partial_x \left(xI-A\right)+J.
$$
Here $J=\textrm{diag}(1,2,\dots, N)$ is diagonal and $A$ only has nonzero entries on the first subdiagonal $A_{j,j-1}=\frac{2\alpha_j}{\alpha_{j-1}}$.
$D$ has the monic MVOP $P_n$ as eigenfunctions with eigenvalue $\Lambda_n= nI+J$.

\paragraph{Remark}
We note that the parameters $\alpha_j$ and $t_j^{(\nu)}$ 
are free in this context, though we do
take them to be positive in order to guarantee that the 
weight is irreducible and positive definite.

In this situation \eqref{eq:eq9} looks very different because the
$n$-dependent part of $\Lambda_n$ commutes with $W$ and hence does not contribute. This also means that $\mathcal{R}$ will not depend on $M$. 
What we are left with is
\begin{equation}
\label{eq:eq9new}
\mathcal{R} W(x)
- 
W(x)
\mathcal{R}^\ast
=
2x[J,W(x)],
\end{equation}
where $[\cdot, \cdot]$ denotes the usual
commutator. The above equation is inherently antisymmetric
and can be reduced to a matrix polynomial equation. We therefore
have a system of $N(N-1)/2$ scalar polynomial equations.

\subsection{$N=2$}
Using the expressions in the appendix we
can write out the $2\times 2$ weight as
$$
W(x)
=
e^{-x^2}t_1
\begin{pmatrix}
1 & 2x  \tfrac{\alpha_2}{\alpha_1} \\
2x  \tfrac{\alpha_2}{\alpha_1}
& \tfrac{t_2}{t_1} + 4x^2 \tfrac{\alpha_2^2}{\alpha_1^2}
\end{pmatrix}.
$$
Then equation \eqref{eq:eq9new} amounts to one independent polynomial equation
that must hold for all $x\in\mathbb{R}$
$$
\mathcal{R}_{21}-\mathcal{R}_{12}\tfrac{t_2}{t_1}
-
2(\mathcal{R}_{11}-\mathcal{R}_{22})\tfrac{\alpha_2}{\alpha_1}x
-
4 \mathcal{R}_{12}\tfrac{\alpha_2^2}{\alpha_1^2}x^2
=
4 \tfrac{\alpha_2^2}{\alpha_1^2}x^2.
$$
This can be easily solved with
$$
\mathcal{R}
=
\begin{pmatrix}
c & -1 \\
-\tfrac{t_2}{t_1} & c
\end{pmatrix},
\qquad
c\in\mathbb{R}.
$$
\subsection{$N>2$}
For higher matrix size $N$ the system of equations becomes 
larger quickly, but not necessarily difficult because 
\eqref{eq:eq9new} is just
an inhomogeneous linear system of equations in the entries 
of $\mathcal{R}$, as mentioned in the Remark on page \pageref{rmk:N}.

Computer algebra calculations for $N\in\{3,4,5,6\}$ indicate that
\eqref{eq:eq9new} has no solution for $\mathcal{R}$.

\subsection*{Funding}
This work was supported by FWO (Research Foundation Flanders, Belgium) [grant number G0C9819N].

\subsection*{Acknowledgements}
The author would like to thank prof. dr. Walter Van Assche and prof.
dr. Erik Koelink for their guidance during this project as well as
prof. dr. Mirta Mar\'{i}a Castro Smirnova for helpful feedback regarding
the introduction of this paper.

\begin{appendix}

\section*{Appendix: Glossary of Explicit Expressions}

The matrix differential operators
that appear in this paper which commute with
their corresponding time-- and band-- limiting
operators can be constructed with the ingredients
listed in this appendix.

\section{Hermite--type}
The matrix weight introduced in \cite[Equation (3.5)]{IKR2} 
is given in LDU
form by (we omit the $\alpha$ superscript)
$$
W^{(\nu)}(x) = L(x)T^{(\nu)}(x)L(x)^\ast,
$$
with $T^{(\nu)}(x)_{jj}=t_j^{(\nu)}e^{-x^2}$ and
$L(x)_{j\geq k} = \frac{\alpha_j}{\alpha_k} \frac{H_{j-k}(x)}{(j-k)!}$,
where $H_n(x)$ denotes the $n$-th standard scalar Hermite polynomial.

For the matrix weight in Section \ref{sec:counter} we can leave the
parameters $t_j^{(\nu)}$ and $\alpha_j$ as free positive real numbers,
but for the Hermite--type example in Section \ref{sec:SP} we need
them to satisfy additional conditions given in \cite[equations (3.7) and (3.9)]{IKR2}.
Three explicit parameter sets that satisfy 
these conditions are given in \cite{IKR2} which we list below.

For $k\in\{1,\dots, N\}$, $\nu>0$, $\lambda>0$:
\begin{align*}
&\begin{cases}
d^{(\nu)} = \frac{1}{\nu+1} , \quad
\alpha_k = \sqrt{2^{1-k}(N-k+1)_{k-1}}, \\
c^{(\nu)} = \frac{\nu}{\nu+1}, \quad
t_k^{(\nu)}  = \frac{(\nu+1)_{k-1}}{(k-1)!} ,
\end{cases}
\\
&\begin{cases}
d^{(\nu)} = \lambda , \quad
\alpha_k = 2^{1-k}\sqrt{(k-1)!(N-k+1)_{k-1}}, \\
c^{(\nu)} = \lambda\nu , \quad
t_k^{(\nu)}  = 2^{-k} \lambda^\nu \Gamma(\nu+k),
\end{cases}
\end{align*}
and for the last set we additionally require
$ \rho>0$ and $C\geq 0$:
$$
\begin{cases}
d^{(\nu)} = \rho  , \qquad \qquad
\alpha_k = 1, \\
c^{(\nu)} = C+\nu\rho , \quad
t_k^{(\nu)}  = \frac{2^{k-1}(\nu+1+C/\rho)_{k-1}}{(k-1)!(N-k+1)_{k-1}}\Gamma(\nu+1+C/\rho). 
\end{cases}
$$
For any of these parameter sets we have
\begin{equation}\label{eq:phihermite}
\Phi^{(\nu)}(x)=
x\phi_1^{(\nu)} + \phi_0^{(\nu)}
\qquad
\Psi^{(\nu)}(x) = 
x\psi_1^{(\nu)} + \psi_0^{(\nu)}
\end{equation}
with
$ \phi_1^{(\nu)} 
=
-d^{(\nu)} A^\ast $,
 $ \phi_0^{(\nu)}
=
d^{(\nu)}(J+\tfrac12 (A^\ast)^2) + c^{(\nu)}I$
and
\begin{align*}
\psi_1^{(\nu)}
&= 
2 (d^{(\nu)}(J - (N + 1) I) - c^{(\nu)}I),
\\
\psi_0^{(\nu)} 
&= 
A^\ast (c^{(\nu)}I + d^{(\nu)} ((N+ 1)I - J)) + 
  \tfrac12 d^{(\nu)}\widetilde{A} J (NI - J),
\end{align*}
with $J = \mathrm{diag}(1,\dots, N)$, $ A_{k,k-1} = 2 \frac{\alpha_k}{\alpha_{k-1}}$ and $\widetilde{A}_{k,k-1}=2\frac{\alpha_{k-1}}{\alpha_k}$,
 $k\in \{ 2,\dots, N\}$.

\section{Laguerre--type}
The matrix weight introduced in \cite[Equation (3.8)]{KR} is 
given in LDU form by (we omit the $\alpha$ superscript)
$$
W^{(\nu)}(x) = L(x)T^{(\nu)}(x)L(x)^\ast,
$$
with $T^{(\nu)}(x)_{jj}= x^{\nu+k}e^{-x} \Delta_{jj}^{(\nu)}=x^{\nu+k}e^{-x}t_j^{(\nu)}$ and
$L(x)_{j\geq k} = \frac{\alpha_j}{\alpha_k} \ell_{j-k}^{(a+k)}(x)$,
where $\ell_n^{(a)}(x)$ denotes the $n$-th standard scalar Laguerre polynomial.
We give three explicit parameter sets that ensure the weight 
satisfies strong Pearson equations as given in \cite{KR},
except the first set which contained a typo for the $t_j^{(\nu)}$.
These hold for $k\in\{1,\dots, N\}$, $\nu>0$, $\lambda>0$:
\begin{align*}
&\begin{cases}
d^{(\nu)} = 1  ,\quad \alpha_k = \sqrt{(N-k+1)_{k}}, \\
c^{(\nu)} = \nu ,\quad t_k^{(\nu)}  = \Gamma(\nu+1)\prod_{s=1}^{k-1}\left(1 +\frac{\nu}{s} \right) ,
\end{cases}
\\
&\begin{cases}
d^{(\nu)} = \lambda ,\qquad  \alpha_k = \sqrt{(k-1)!(N-k+1)_{k-1}}, \\
c^{(\nu)} = \lambda\nu ,\quad t_k^{(\nu)}  = \lambda^\nu \Gamma(\nu+k) ,
\end{cases}
\end{align*}
and for the last set we additionally require
$\rho>0$ and $C\geq 0$:
$$
\begin{cases}
d^{(\nu)} = \rho  ,\qquad\qquad \alpha_k = 1, \\
c^{(\nu)} = C+\nu\rho ,\quad t_k^{(\nu)}  = \frac{(\nu+1+C/\rho)_{k-1}}{(k-1)!(N-k+1)_{k-1}}\rho^\nu \Gamma(\nu+1+C/\rho) .
\end{cases}
$$
For these parameters we have
\begin{equation}\label{eq:philaguerre}
\Phi^{(\nu)}(x)= x^2 \phi_2^{(\nu)} + x\phi_1^{(\nu)},
\qquad
\Psi^{(\nu)}(x) = x\psi_1^{(\nu)} + \psi_0^{(\nu)}
\end{equation}
with
$
\phi_2^{(\nu)}
=
-d^{(\nu)}\left(L(0)^\ast \right)^{-1} A^\ast L(0)^\ast
$, 
$
\phi_1^{(\nu)}
=
d^{(\nu)}\left(L(0)^\ast \right)^{-1} J L(0)^\ast
+
c^{(\nu)}I,
$
and
\begin{align*}
\psi_1^{(\nu)}
&= 
d^{(\nu)}\left(L(0)^\ast \right)^{-1} (J-A^\ast(J+(\nu+1)I)) L(0)^\ast
-(d^{(\nu)}(N+1)+c^{(\nu)})I,
\\
\psi_0^{(\nu)} 
&= 
\left(L(0)^\ast \right)^{-1}\left( 
(J+(\nu+1)I)(d^{(\nu)}J+c^{(\nu)}I)
+
\Delta^{(\nu)-1} A \Delta^{(\nu+1)} \right)L(0)^\ast 
\end{align*}
with  $J = \mathrm{diag}(1,\dots, N)$ and $A_{k,k-1} = - \frac{\alpha_k}{\alpha_{k-1}}$.

\section{Gegenbauer--type}
The matrix weight introduced in \cite[Theorem 2.2]{KdlRR} is given in LDU
form by
$$
W^{(\nu)}(x) = L^{(\nu)}(x)T^{(\nu)}(x)L^{(\nu)}(x)^\ast,
$$
with $T^{(\nu)}(x)_{jj}=t_j^{(\nu)}(1-x^2)^{\nu+j-1/2}$ and
$L(x)_{j\geq k} = \beta_{j,k}^{(\nu)} C_{j-k}^{(\nu+k)}(x)$,
where $C_n^{(\nu)}(x)$ denotes the standard scalar Gegenbauer polynomials.
Note however that we have adhered to the index 
notation used in \cite{KdlRR}, which is different from all the
other cases described in this paper.
The matrix size is $N=2\ell+1$ with $\ell\in \tfrac12\mathbb{N}$
and the matrix index takes values $j\in \{0, \dots, 2\ell\}$.

We give the parameters that appear in the weight
$\beta_{j,k}^{(\nu)}=\frac{j!}{k!(2\nu+2k)_{j-k}}$ and 
$$
t_k^{(\nu)}
=
\frac{k!(\nu)_k}{(\nu+1/2)_k}
\frac{(2\nu+2\ell)_k(2\ell+\nu)}
{(2\ell-k+1)_k(2\nu+k-1)_k},
$$
as well as another parameter that will
allow for more compact expressions
$$
c^{(\nu)}
=
\frac{(2\nu+1)(2\ell++\nu+1)\ell^2}{\nu(2\nu+2\ell+1)(2\ell+\nu)(\ell+\nu)}.
$$
To the same end we give the following 
diagonal matrices
$$
J=\mathrm{diag}(0,\dots, 2\ell),
\quad
A_{j,j-1} = 1,
\quad
K_n^{(\nu)}
=
\frac{2\ell+2\nu+n}{-\ell^2}
(J+\nu I)((2\ell+\nu)I-J),
$$
the first two of which do not appear in \cite{KdlRR}.
We should also note that our $
\Phi^{(\nu)}(x) = x^2\phi_2^{(\nu)}+x\phi_1^{(\nu)}+\phi_0^{(\nu)}
$ and $
\Psi^{(\nu)}(x) = x\psi_1^{(\nu)}+\psi_0^{(\nu)}
$
follow a slightly different
convention than in \cite[Theorem 2.4]{KdlRR} due to a difference
in the form of the strong Pearson equations. 
We use the coefficients
\begin{equation}\label{eq:phigegenbauer}
\begin{aligned}
\phi_2^{(\nu)} &= \tfrac{c^{(\nu)}}{2\ell+2\nu+1}K_1^{(\nu)},
\qquad
\psi_1^{(\nu)} = c^{(\nu)} K_1^{(\nu)} \\
\phi_1^{(\nu)}
&=
\tfrac{c^{(\nu)}}{2\ell^2}\Bigl(
((2\ell+1)I-2J)(J-(2\ell+1)I)A
+
((2\ell-1)I-2J)A^\ast J
\Bigr) \\
\phi_0^{(\nu)}
&=
\tfrac{c^{(\nu)}}{4\ell^2}
\Bigl(
4(\ell+v)^2I
+
((2\ell+2)I-J)((2\ell+1)I-J)A^2
\\
&\hspace{5cm	}+
2J^2 -4\ell J+2\ell I
+ (A^\ast J)^2
\Bigr) \\
\psi_0^{(\nu)}
&=
c^{(\nu)}\tfrac{2\ell+1+2\nu}{-2\ell^2}
\left(
A(J-2\ell I)(J+\nu I )
-
A^\ast J ((2\ell+\nu)I-J)
\right)
\end{aligned}
\end{equation}
The only difference with \cite{KdlRR} is that here we have absorbed the
factor $c^{(\nu)}$ into the coefficients.

Lastly we note that the second term in $\psi_0^{(\nu)}$ has an errant opposite 
sign in \cite{KdlRR}.

\section{Charlier--type}
The matrix weight introduced in \cite{EMR} is given in LDU
form by
\begin{equation*}
W^{(\nu)}(x)
=
(I+A)^{x+\nu}T^{(\nu)}(x)(I+A^\ast)^{x+\nu},
\qquad
 x,\nu\in \mathbb{N}_0,\quad a>0,
\end{equation*}
with the parameters $t_j^{(\nu)} >0$, $\alpha_j >0$ in
the diagonal matrices
$$
A_{j,k}
=
\begin{cases}
             \frac{\alpha_j }{\alpha_{j-1}}, &  j=k+1 \\
             0, &   j\neq k+1
\end{cases},
\qquad
T^{(\nu)}(x) = \frac{a^x}{x!} \textrm{diag}(t_1^{(\nu)}, \dots, t_N^{(\nu)}).
$$
Here the indices take values $j,k\in \{1,\dots , N\}$ again.

The following parameter values 
\begin{equation*}
\left(\frac{\alpha_j}{\alpha_{k}}\right)^2 
= 
a^{k-j} \frac{(N-k)!}{(N-j)!}
,
\qquad
t_k^{(\nu)}
= \left(\frac{a}{2}\right)^\nu (k)_\nu,
\end{equation*}
ensure that the weight satisfies the discrete strong 
Pearson equations. The matrix polynomials that arise from those
equations
\begin{align*}
\Phi^{(\nu)}(x)
&=
W^{(\nu)}(x)^{-1}W^{(\nu+1)}(x), 
\\
\Psi^{(\nu)}(x)
&=
W^{(\nu)}(x)^{-1}(W^{(\nu+1)}(x)-W^{(\nu+1)}(x-1)),
\end{align*}
are then of degree two and one respectively. 
In particular we have
\begin{equation}\label{eq:phicharlier}
\Phi^{(\nu)}(x) = x^2 \phi_2^{(\nu)} + x \phi_1^{(\nu)} +\phi_0^{(\nu)}, \qquad \Psi^{(\nu)}(x) = x \psi_1^{(\nu)} + \psi_0^{(\nu)},
\end{equation}
where $\phi_2^{(\nu)} =-\tfrac{1}{2}A^{*}(A^{*}+I)^{-1},$
\begin{align*}
\phi_1^{(\nu)} &=  \tfrac12 \left( 2J -(N+1)I -aA^*-(2\nu+1)A^\ast (A^*+I)^{-1} \right),\\
\phi_0^{(\nu)}  &= \psi_0^{(\nu)} = (A^{*}+I)^{-\nu} (T^{(\nu)}(0))^{-1}(A+I)T^{(\nu+1)}(0)(A^{*}+I)^{\nu+1},\\
\psi_1^{(\nu)} 
&= 
\tfrac{1}{2}(J-(N+1+\nu)I-aA^*-(\nu+1)A^*(I+A^*)^{-1}).
\end{align*}
\paragraph{Remark}
To see that this weight is in fact very similar to the
Hermite--, Laguerre-- and Gegenbauer--type weights we
have also seen in this paper, we can also define
\begin{equation*}
L(x)_{j,k}=  (-a)^{j-k}\frac{\alpha_j }{\alpha_k}
             \frac{c^{(a)}_{j-k}(x)}{(j-k)!}, \quad j\geq k, 
             \qquad \qquad    L(x)_{j,k} = 0, \quad   j<k.
\end{equation*} 
Since $L(x)$ satisfies $L(x+1)=L(x)(I+A)$, due to ladder relations
of the scalar Charlier polynomials, we have
$L(x)=L(0)(I+A)^x = (I+A)^x L(0)$ for  $x\in\mathbb{Z}$. So then it is
clear that the Charlier--type weight in this appendix
is congruent to a weight of the form
$L(x)T^{(\nu)}(x)L(x)^\ast = L(-\nu) W^{(\nu)}(x)L(-\nu)^\ast$ which looks
more similar to the other weights in this chapter.

\end{appendix}

\DeclareRobustCommand{\VAN}[3]{#3}
\DeclareRobustCommand{\DELA}[3]{#3}

\bibliographystyle{plainnat}

\end{document}